\documentclass[11pt,a4paper]{article}
\usepackage{epsf,algorithm,epsfig,amsfonts,amsgen,amsmath,amstext,amsbsy,amsopn,amsthm,lineno}
\usepackage{color}
\usepackage{tikz}
\usepackage{cite}
\usepackage{enumerate}
\usepackage{float}
\usepackage{epstopdf}
\usepackage{subcaption}
\usepackage{multicol}
\usepackage{amssymb}
\usetikzlibrary{decorations.pathreplacing,calligraphy}
\usetikzlibrary{calc}

\newcommand{\cH}{\mathcal{H}}
\newcommand{\bP}{\mathbb{P}}

\newtheorem{theorem}{Theorem}

\newtheorem{lemma}{Lemma}

\newtheorem{problem}{Problem}
\newtheorem{proposition}{Proposition}

\newtheorem{claim}{Claim}

\def\final{0}  % set this to 1 to get a comment-free version
\def\iflong{\iffalse}
\ifnum\final=0  %namely if we allow comments in the output
\newcommand{\znote}[1]{{\color{red}[{\tiny \textbf{Zixuan:} \bf #1}]\marginpar{\color{red}*}}}
\newcommand{\gnote}[1]{{\color{blue}[{\tiny \textbf{Gyula:} \bf #1}]\marginpar{\color{blue}*}}}
\newcommand{\ynote}[1]{{\color{purple}[{\tiny \textbf{Yuhang:} \bf #1}]\marginpar{\color{purple}*}}}
\else % in this case [final=1] we don't want any comments to show
\newcommand{\znote}[1]{}
\newcommand{\gnote}[1]{}
\newcommand{\ynote}[1]{}
\fi

\newcounter{mathitem}

\usetikzlibrary{decorations.markings}
\tikzstyle{vertex}=[circle, draw, inner sep=0pt, minimum size=5pt]

\begin{document}

\title{\bf\Large Most probably trangle-free graphs %\thanks{This research was supported by National Key Research and Development Program of China (No.~2023YFA1010203), National Natural Science Foundation of China (Nos. ~12401464, 12271425, 12131013 and 12471334) and Shaanxi Fundamental Science Research Project for Mathematics and Physics (No. 22JSZ009).}
}
\date{}
\author{Yuhang Bai$^{a,b}$,  Gyula O.H. Katona$^{c}$, Zixuan Yang$^{a,b}$\thanks{Corresponding author.}
~\\[2mm]
\small $^{a}$School of Mathematics and Statistics, \\
\small Northwestern Polytechnical University, Xi'an, Shaanxi, P.R. China\\
\small $^{b}$Xi'an-Budapest Joint Research Center for Combinatorics, \\
\small Northwestern Polytechnical University, Xi'an, Shaanxi, P.R. China\\
\small $^{c}$HUN-REN Alfréd Rényi Institute of Mathematics, \\
\small Hungarian Academy of Sciences, Budapest, Hungary\\
}
\maketitle

\begin{abstract}

The celebrated Mantel's theorem states that any  triangle-free graph on $n$ vertices contains  at most $\left\lfloor n^2/4\right\rfloor$ edges.  It is natural  to ask how many triangles must exist in a graph with more than $\left\lfloor n^2/4\right\rfloor$ edges--a problem known as the Erd\H{o}s-Rademacher problem. In this paper, we propose a probabilistic variant of this classic problem. Specifically, given an $n$-vertex graph $G$ with  $\left\lfloor n^2/4\right\rfloor+i$ ($i>0$) edges, we choose the edges of $G$ independently  with probability $p$, and the resulting new graph is triangle-free with a certain probability. Our goal is to maximize this probability by choosing  $G$ appropriately.  For the case where $G$ has $ \left\lfloor n^2/4\right\rfloor +1$ edges, we determine the exact  maximum probability.

%A graph is called $F$-free if it does not contain the graph $F$ as a subgraph.  In 1907, Mantel established a classic  result in extremal graph theory: every triangle-free graph $G$ on $n$ vertices satisfies $|E(G)|\le \left\lfloor n^2/4\right\rfloor$, where $E(G)$ denotes the edge set of $G$. Suppose that $|E(G)|= \left\lfloor n^2/4\right\rfloor +i$. We randomly select edges of $G$ independently each with probability $p$, and the resulting new graph is triangle-free with a certain probability. Our goal is to maximize this probability by choosing  $G$ appropriately. In this paper, we determine the exact maximum  for  a triangle-free graph with $|E(G)|= \left\lfloor n^2/4\right\rfloor +1$.  

\medskip
\noindent {\bf Keywords:} Erd\H{o}s-Rademacher problem; probabilistic supersaturation; triangle-free graph
\medskip

\noindent {\bf 2020 MSC Codes:} 05C35; 05D40; 05C80
\smallskip
\end{abstract}

%%%%%%%%%%%%%%%%%%%%%%%%%%%%%%%%%%%%%%%%%%%%%%%%%%%%%%%%

\setcounter{footnote}{0}
\renewcommand{\thefootnote}{}
\footnotetext{E-mail addresses:  {\tt yhbai@mail.nwpu.edu.cn (Y. Bai), katona.gyula.oh@renyi.hu (G. Katona), yangzixuan@nwpu.edu.cn (Z. Yang)}}

%%%%%%%%%%%%%%%%%%%%%%%%%%%%%%%%%%%%%%%%%%%%%%%%%%%%
\section{Introduction}
%%%%%%%%%%%%%%%%%%%%%%%%%%%%%%%%%%%%%%%%%%%%%%%%%%%%

In extremal graph theory, an important series of problems investigate the maximum number of edges in a graph under a given set of constraints.  A graph is called \emph{$F$-free} if it does not contain the graph $F$ as a subgraph. The most well-known result in extremal graph theory is the following theorem obtained by  Mantel \cite{Man} in 1907.  This fundamental theorem has since inspired a great number of extensions and variations. 

\begin{theorem}[Mantel, \cite{Man}]\label{Man}
Let $G$ be an $n$-vertex triangle-free graph. Then  $G$ has at most $\left\lfloor n^2/4\right\rfloor$ edges.
\end{theorem}

A  trend in extremal combinatorics is to study the supersaturation extension of
classic results. This problem, sometimes referred to as the Erd\H{o}s–Rademacher problem,
asks for the number of forbidden substructures that must appear in a configuration larger
than the extremal threshold.  The first such line of research extended Mantel’s theorem by Rademacher  (an unpublished result from 1941, see \cite{Erd3}), who showed that any graph with $\left\lfloor n^2/4\right\rfloor+1$ edges must contain at least $\left\lfloor n/2 \right\rfloor$ triangles. Determining the number of triangles in graphs of higher edge density attracted a great deal of attention, starting with the works of Erd\H{o}s \cite{Erd1,Erd2}.  Lov\'{a}sz and Simonovits \cite{Lov} further advanced this work in 1983, establishing  the following result. This research direction culminated in an asymptotic solution by Razborov \cite{Raz} in 2008, and its exact solution was established by Liu, Pikhurko, and Staden \cite{Liu} in 2020.

\begin{theorem}[Lovász and Simonovits, \cite{Lov}]\label{lem:count}
Let  $G$ be a graph with  $n$ vertices and  $\left\lfloor n^2/4 \right\rfloor + i$ edges, where $i \le n/2$. Then $G$  contains at least $i\left\lfloor n/2\right\rfloor$ triangles.
\end{theorem}

In this paper, we introduce a probabilistic measure of supersaturation for graphs, in parallel with  Katona, Katona, and Katona's results on  probabilistic supersaturation  for intersecting families \cite{Kat}.  
Rather than minimizing the total number of triangles in large graphs, we seek to maximize the probability of a random subgraph being triangle-free. More specifically speaking, for an $n$-vertex graph $G$, we choose the edges of $G$ independently with probability $p$ ($0<p<1$) and delete them with probability $1-p$.  Let $G_p$ denote the random graph  obtained in this way, and let $\bP(\cdot)$ represent the probability of an event. We want to maximize the probability of the event that $G_p$ is triangle-free, for graphs $G$ of given the number of edges. 

Clearly, if a graph $G$ is triangle-free, then $G_p$ must also be triangle-free, and hence one should
take a triangle-free graph if possible. Thus, analogous to the supersaturation extension of Mantel’s theorem \cite{Man}, our focus is on determining the maximum probability that a graph exceeding the extremal bound remains triangle-free.

Notably this probabilistic problem is in fact stronger than the counting version described before. 
Indeed, by conditioning on the number of edges in $G_p$, we  apply the law of total probability to obtain that

\begin{align}\label{tf}
\mathbb{P}(G_p \text{ is triangle-free}) &= \sum_{k=0}^m \mathbb{P}(G_p \text{ is triangle-free} \mid |E(G_p)|=k) \mathbb{P}(|E(G_p)|=k) \notag\\
&= \sum_{k=0}^m \bigg(\text{tf}(G,k)/\binom{m}{k}\bigg)\bigg(\binom{m}{k}p^k (1-p)^{m-k}\bigg)\notag\\
&= \sum_{k=0}^m \text{tf}(G,k) p^k (1-p)^{m-k},
\end{align}
where $m$ is the number of edges of $G$, $\text{tf}(G,k)$ denotes the number of triangle-free subgraphs  with $k$ edges of $G$, and $t(G)$ denotes the number of triangles in $G$.
Note that $\text{tf}(G,0) =1$, $\text{tf}(G, k) = \binom{m}{k}$ for $k\in \{1,2\}$, and $\text{tf}(G,k) < \binom{m}{k}$ for all $k \ge 3$ and $t(G)\neq 0$ since any triangle in $G$ corresponds to a $3$-edge subset that is not triangle-free.
If we take $p = o(1/m^3)$, then $m^3 p^3 = o(p^2)$, which implies that higher-order terms ($k \ge 4$) can be neglected. So expanding the first few terms of the sum on the right-hand side of (\ref{tf}) gives
\begin{align*}
\mathbb{P}(G_p \text{ is triangle-free}) =& (1-p)^m + m p (1-p)^{m-1} + \binom{m}{2} p^2 (1-p)^{m-2} \\
&- t(G) p^3 (1-p)^{m-3} + o(p^2).
\end{align*}
This quantity is maximized if and only if the number of triangles in $G$ is minimized, which is precisely the objective of the counting supersaturation problem.
Hence a solution to the probabilistic problem for all values of $p$ provides a solution to the counting problem as well.

Define 
$$
\Phi_p(G)=\mathbb P(G_p ~\text{is triangle-free}),
$$
and
$$
T(i)=\max\{\Phi_p(G): |V(G)|=n, |E(G)|=\left\lfloor\frac{n^2}{4}\right\rfloor+i, i>0\}.
$$
The graphs attaining this maximum can be called \emph{the most probably triangle-free graphs}.
Determining the value of $T(1)$ is our main result in this paper.

\begin{theorem}\label{thm:triangle}
Let $n\ge 3$ be an integer, and let $p\in(0,1)$. Then
\begin{align*}
T(1)=1-p+p(1-p^2)^{\left\lfloor \frac{n}{2}\right\rfloor}
\end{align*}
with equality for the graphs obtained by adding a single edge within   the partition  of size $\left\lceil n/2\right\rceil$ of the complete  bipartite graph $K_{\left\lceil \frac{n}{2}\right\rceil,\left\lfloor \frac{n}{2}\right\rfloor}$.
\end{theorem}

%%%%%%%%%%%%%%%%%%%%%%%%%%%%%%%%%%%%%%%%%%%%%%%%%%%%
\section{Preliminaries}
%%%%%%%%%%%%%%%%%%%%%%%%%%%%%%%%%%%%%%%%%%%%%%%%%%%%

In this section, we present some basic definitions and notation in hypergraphs and probability, and introduce the Fortuin-Kasteleyn-Ginibre (FKG) inequality \cite{For} required for our subsequent proofs.

A \emph{hypergraph} $\mathcal{H}$ is a pair $\mathcal{H} = (V , E)$, where $V:=V (\mathcal{H})$ is a set of vertices and $E:=E(\mathcal{H})$ is a set of nonempty subsets of $V$. A hypergraph is \emph{k-uniform} if $E \subseteq $$V \choose{ k}$, where ${V \choose k}:= \{ T \subseteq V : | T| = k\}$. In particular, $2$-uniform hypergraphs are  graphs, which we refer to as such for simplicity. An \emph{(weakly) independent set} of a hypergraph $\mathcal{H}$ is a subset of $V(\mathcal{H})$ that contain no edges. A $k$-uniform hypergraph is \emph{linear} if any two edges intersect in at most one vertex.
For $e\in  E(\mathcal{H})$, let $\mathcal{H}-\{e\}$ denote the $r$-graph obtained from $\mathcal{H}$ by removing the edge $e$.

Let $\Omega=\{0,1\}^n$ and $p \in (0, 1)$. We sample $x=(x_1,\dots,x_n)\in\Omega$ by choosing each coordinate
$x_i\in\{0,1\}$ independently, with $\bP(x_i=1)=p$, $\bP(x_i=0)=1-p$.
Equivalently, the induced product measure on $\Omega$ is $\bP(x)=\prod_{i=1}^n p^{x_i}(1-p)^{1-x_i}$.
This probability space is called the \emph{independent Bernoulli space $\{0,1\}^n$}.
A partial order on the elements of  $\Omega$ is defined as follow:
\[
(x_1, x_2, \dots, x_n) \geq (y_1, y_2, \dots, y_n)  \quad \text{if and only if} \quad x_i \geq y_i \quad \text{for all}  \quad1 \leq i \leq n.
\]
Let $A \subseteq \Omega$ be an event. 
We say that $A$ is \emph{decreasing} if $(x_1, x_2, \dots, x_n) \in A$ and  $ (y_1, y_2, \dots, y_n) \leq (x_1, x_2, \dots, x_n)$ implies that $ (y_1, y_2, \dots, y_n) \in A$. By the  definition of decreasing event, one can get the following result.

\begin{proposition}\label{prop}
Let $\{A_1,A_2,\ldots,A_m\}$ be a finite collection of decreasing events in $\Omega$. Then the intersection
\[
A = \bigcap_{i\in [m]} A_i
\]
is also a decreasing event.
\end{proposition}

Our proof uses a special case of  Fortuin-Kasteleyn-Ginibre (FKG) inequality \cite{For}, dating back to Harris \cite{Har}.  The original version of the theorem below is stated for two decreasing events, however, by Proposition \ref{prop}, it can be shown that the Harris's inequality \cite{Har} holds in the following more general form.

\begin{theorem}[Harris, \cite{Har}]\label{Har-ineq}
    Let $A_1, A_2,\ldots, A_m \subseteq \Omega$ be events in the independent Bernoulli space $\{0,1\}^n$.
    If $A_1,A_2,\ldots,A_m$  are all decreasing, then
    \[
    \mathbb{P}(\bigcap_{i\in [m]}A_i ) \geq \prod_{i\in [m]}\mathbb{P}(A_i).
    \]
\end{theorem}

%%%%%%%%%%%%%%%%%%%%%%%%%%%%%%%%%%%%%%%%%%%%%%%%%%%%
\section{Proof of Theorem~\ref{thm:triangle}}
%%%%%%%%%%%%%%%%%%%%%%%%%%%%%%%%%%%%%%%%%%%%%%%%%%%%

To complete the proof of Theorem \ref{thm:triangle}, we first prove the following lemma, which is used to establish the upper bound in our main result.

\begin{lemma}\label{lem:linear3}
Let $\cH$ be a linear $3$-uniform hypergraph with vertex set $V(\cH)$ and edge set $E(\cH)$ where $|V(\cH)| = n$ and $|E(\cH)| = r$. Let $S\subseteq V(\cH)$
be formed by including each vertex independently with probability $p\in(0,1)$. Then
$$
\mathbb P(S \text{ is an independent set in } \cH)\le 1-p+p(1-p^2)^r.
$$
\end{lemma}

\begin{proof}
We prove the result by using induction on the number of edges $r$.
% For  the case of $r=1$, let $e=\{x,y,z\}$ be the edge of $\cH$. It is clear that every subset $S\subseteq V(\cH)$ which does not contain at least one of $\{x,y,z\}$ is an independent set, and the corresponding probability equals $1-p^3$. Thus the result holds. 
For the case of $r = 0$, It is clear that the result holds. 

Next we may assume that $r\ge 1$, and the statement holds for $r-1$. 
 Let $\cH'= \cH -\{e_0\}$, where  $e_0=\{x_0,y_0,z_0\}$ is an edge in $ \cH$.
We  define two events as follow.
\begin{itemize}
\item Event $A$: the subset $S$  is an independent set in $\cH'$.
\item Event $B$: the subset $S$ contains the vertices $x_0,y_0,$ and $z_0$.
\end{itemize}
Note that $S$ is an independent set in $\cH$ if and only if the event $A$ occurs and the event $B$ does not occur. So we have
\begin{align}\label{S-ind}
\mathbb P(S \text{ is an independent set in } \cH)&=\mathbb P(A\cap \overline{B})\notag\\
&=\mathbb P(A)-\mathbb P(A\cap B)\notag\\
&=\mathbb P(A)-p^3\,\mathbb P(A\mid B).
\end{align}

\begin{claim}\label{cl:condtiona}
$\mathbb{P}(A\mid B)\ge (1-p^2)^{r-1}$.
\end{claim}

\begin{proof}
We identify each subset $S\subseteq V(H)=\{v_1,\dots,v_n\}$ with its characteristic vector
$\omega_S\in\Omega=\{0,1\}^n$, where
\[
\omega_S(v_i)=
\begin{cases}
1 & \text{if } v_i\in S,\\
0 & \text{if } v_i\notin S.
\end{cases}
\]
There is a partial ordering on the vectors of $\Omega$:  $\omega\le \omega'$ if for all $v_i$ we have 
$\omega(v_i)\le \omega'(v_i)$. On the level of sets this  means that $\omega_S\le \omega_{S'}$ if and only if $S\subseteq S'$. 
For each edge $e\in E(\cH')$, we define
\begin{itemize}
\item event $A_e$: the subset $S$  does not contain all three vertices of $e$.
\end{itemize}
For  $S'\subseteq S$, since $\omega_S \in A_e$   implies that $\omega_{S'} \in A_e$, it follows from the definition that $A_e$ is a decreasing event. Then by the definition of event $A$, we have 
\begin{align}\label{AAe}
A=\bigcap_{e\in E(\cH')}A_e.
\end{align}

We assert that $\mathbb P(A_e\mid B)\ge 1-p^2$ for all $e\in E(\cH')$. Since $\cH$ is linear, every edge $e \in E\left(\cH^{\prime}\right)$ satisfies $\left|e \cap e_0\right| \leq 1$, where $e_0=\{x_0,y_0,z_0\}$.  If $|e \cap e_0|=0$,  then we have 
$$
\mathbb{P}(\overline{A_e} \mid B)=p^3,
$$
and thus 
\begin{align*}
\mathbb{P}\left(A_e \mid B\right)=1-p^3 > 1-p^2,
\end{align*}
as desired. So we consider the case when $\left|e \cap e_0\right|=1$. Without loss of  generality, let $x_0$ be the common vertex of $e\cap e_0$. The event $B$ implies that  $x_0 \in S$, and the other two vertices of $e$ are still included independently with probability $p$. This implies that
\begin{align*}
\mathbb{P}(\overline{A_e} \mid B)=p^2, 
\end{align*}
and thus 
\begin{align*}
 \mathbb{P}\left(A_e \mid B\right)=1-p^2,
 \end{align*}
the assertion holds.

Condition on the event $B$, which by definition forces $x_0,y_0,z_0\in S$.
Under this conditioning, the inclusion of each vertex $v\in V\setminus\{x_0,y_0,z_0\}$ is still decided
independently with probability $p$. 
Thus, under the conditional distribution $\bP(\,\cdot\,\mid B)$, we may apply Harris's inequality (i.e., Theorem \ref{Har-ineq}) and equality~(\ref{AAe}) to derive 
\begin{align*}
\mathbb P(A\mid B)&=\mathbb P\Big(\bigcap_{e\in E(\mathcal H')}A_e \Big| B\Big)\\
&\ge \prod_{e\in E(\mathcal H')}\mathbb P(A_e\mid B)
\ge (1-p^2)^{r-1},
\end{align*}
which completes the proof of the claim.
\end{proof}

The induction hypothesis applied to $\cH'$ yields
$$
\mathbb P(A)\le 1-p+p(1-p^2)^{r-1}.
$$
Thus, by Claim~\ref{cl:condtiona} and (\ref{S-ind}),  we have 
\begin{align*}
\mathbb P(S \text{ is an independent set  in } \cH)&=\mathbb P(A)-p^3\,\mathbb P(A\mid B)\\
&\le \mathbb P(A)-p^3(1-p^2)^{r-1}\\
&\le 1-p+p(1-p^2)^{r-1}-p^3(1-p^2)^{r-1}\\
&=1-p+p(1-p^2)^r.
\end{align*}
This completes the proof of Lemma~\ref{lem:linear3}.
\end{proof}

\noindent \emph{Proof of Theorem~\ref{thm:triangle}.} 
We will complete the proof  by establishing the lower bound  
\begin{align*}
T(1)\ge1-p+p(1-p^2)^{\left\lfloor n/2\right\rfloor}
\end{align*}
 and the upper bound 
 \begin{align*}
 T(1)\le1-p+p(1-p^2)^{\left\lfloor n/2\right\rfloor}.
 \end{align*} 

We first show the lower bound. Let $K_{a,b}$ be a complete bipartite graph with the parts $A$ and $B$ such that $|A|=a=\left\lfloor n/2\right\rfloor$ and $|B|=b=\left\lceil n/2\right\rceil$. Let $G$ be the graph obtained from $K_{a,b}$ by adding a single edge $uv$ inside $B$. One can see that $|V(G)|=a+b=n$ and $|E(G)|=ab+1$.

We consider $\Phi_p(G)$ by splitting it into the cases $uv\notin E(G_p)$ and $uv\in E(G_p)$.

(a) For $uv\notin E(G_p)$,  we have that
$G_p$ is a bipartite graph, which implies that  $G_p$ is triangle-free. This occurs with probability
\begin{align}\label{p2}
\mathbb P\bigl(uv\notin E(G_p)\bigr)=1-p.
\end{align}

(b) For $uv\in E(G_p)$, we have that every triangle $T_k$ in $G$ has the vertex set $\{x_k,u,v\}$ and the edge set $\{uv,ux_k,vx_k\}$, where $x_k\in A$.
For any fixed $x_k\in A$, the edges $x_ku$ and $x_kv$ exist in $G_p$ independently with probability $p$. This means  that the probability of event that the triangle $T_k$ does not exist in $G_p$ is $1-p^2$. All these events are independent over the different vertices $x_k\in A$, where $1\le k\le a$. Hence
\begin{align}\label{p3}
\mathbb P\bigl(G_p \text{ is triangle-free}\mid uv\in E(G_p)\bigr)=(1-p^2)^a.
\end{align}
By (\ref{p2}) and (\ref{p3}), we have
\begin{align*}
\Phi_p(G)&=\mathbb P\bigl(uv\notin E(G_p)\bigr)+\mathbb P\bigl(uv\in E(G_p)\bigr)\mathbb P\bigl(G_p \text{ is triangle-free}\mid uv\in E(G_p)\bigr)\\
&=(1-p)+p(1-p^2)^a,
\end{align*}
which implies
$$
T(1)\ge \Phi_p(G)=(1-p)+p(1-p^2)^a.
$$

Next we prove the upper bound.
Let $G$ be an arbitrary graph with $|V(G)|=n$ and $|E(G)|=ab+1$, and let $t(G)$ denote the number of triangles in $G$.
We define a $3$-uniform hypergraph $ H$ based  on $G$ with 
\begin{itemize}
\item the vertex set $V(\cH)=E(G)$ 
\item and the edge set $E(\cH)$: each triangle of $G$ corresponds to a hyperedge  $\cH$ in composed from its three edges.
\end{itemize}
Since any two distinct triangles in $G$ share at most one edge,  
it follows that any two hyperedges in $\cH$ intersect in at most one vertex. Thus $\cH$ is  a linear $3$-uniform hypergraph.

Let $S:=E(G_p)$. Then $S\subseteq E(G)=V(\cH)$ is obtained by choosing each vertex
independently with probability $p$. Moreover, $G_p$ is triangle-free 
if and only if $S$ does not contain edges of $\cH$, i.e. $S$ is an independent set in $\cH$. Hence
$$
\Phi_p(G)=\mathbb P(S \text{ is an independent set in } \cH).
$$

Recall that  $\cH$ is a linear 3-uniform hypergraph with $|E(\cH)|=t(G)$. So by Theorem~\ref{lem:count} and Lemma~\ref{lem:linear3}, we have
\begin{align*}
\Phi_p(G)&=\mathbb P(S \text{ is an independent set in } \cH)\\
&\le 1-p+p(1-p^2)^{t(G)}\\
&\le 1-p+p(1-p^2)^a.
\end{align*}
This completes the proof of Theorem \ref{thm:triangle}. \qed

\section{Conclusions}
\label{sec:conc}
%%%%%%%%%%%%%%%%%%%%%%%%%%%%%%%%
In this paper, we introduce a probabilistic version of the graph supersaturation problem. Instead of forbidding a subgraph deterministically,  we consider $n$-vertex graphs with a prescribed number of edges and ask which ones maximize the probability that their random $p$-subgraph contains no copy of a given forbidden configuration. Our main result settles the triangle case at one edge above Mantel's bound and provides an exact formula for $T(1)$.

We close the paper by mentioning two open problems. Perhaps the most natural one is to determine $T(i)$ for general $i$, as follows.
\begin{problem}\label{prob:1}
    Determine $T(i)$ when $i > 1$.
\end{problem}

Our result for $i=1$ holds for every $p\in(0,1)$, whereas for $i>1$ we expect the extremal graphs to depend on $p$.
For example, when $(n,i)=(6,2)$, define three candidate graphs as follows. Let $G^1$ be obtained from $K_{3,3}$ by adding a $2$-edge star in one part, and let $G^2$ be obtained by adding one edge in each part, and let $G^3$ be obtained from $K_{2,4}$ by adding a $P_4$ inside the part of size $4$. Recall that for very small $p$, the probability that $G_p$ is triangle-free is maximized if and only if
the number of triangles in $G$ is minimized. Consequently, any extremal graph must be one of
$G^1,G^2,G^3$. However, $G^3$ is never extremal, since a direct computation shows that
\[
\mathbb{P}(G^2_{p}\text{ is triangle-free})
-\mathbb{P}(G^3_{p}\text{ is triangle-free})
>0.
\]
Hence it suffices to compare $G^1$ and $G^2$. Furthermore, a direct computation shows that
\begin{align*}
\mathbb{P}(G_{p}^1 \text{ is triangle-free})
&= (1-p)^2+2p(1-p)(1-p^2)^3+p^2(1-p+p(1-p)^2)^3\\
&=1-6p^3+9p^5+6p^6-14p^7-3p^8+12p^9-6p^{10}+p^{11},
\end{align*}
whereas
\begin{align*}
\mathbb{P}(G_{p}^2\text{ is triangle-free})
=&(1-p)^2+2p(1-p)(1-p^2)^3+p^2 (1-p^2)^2\bigl((1-p)^4\\
&+4p(1-p)^3+2p^2(1-p)^2\bigr)\\
=&1-6p^3+10p^5+2p^6-10p^7+4p^9-p^{10}.
\end{align*}
Moreover,
\[
\mathbb{P}(G_{p}^1\text{ is triangle-free})
-\mathbb{P}(G_{p}^2\text{ is triangle-free})
=-p^5(1-p)^3\,(p^3-2p^2-p+1).
\]
Let $p_0\in(0,1)$ be the  root of $p^3-2p^2-p+1=0$. 
Then for $0<p<p_0$,
\[
\bP \left(G^2_p \text{ is triangle-free}\right) > \bP \left(G^1_p \text{ is triangle-free}\right),
\]
 and the inequality reverses for $p_0<p<1$.

It is natural to pose a second question regarding the graphs that maximize the probability of the event that $G_p$ is $K_k$-free, where $K_k$ denotes the clique of order $k$. 	
\begin{problem}\label{prob:2}
    Determine 
    \begin{align*}
    \max\{\mathbb{P}(G_p \text{ is $K_k$-free } ): |V(G)|=n, |E(G)|=\mathrm{ex}(n, K_k) + i, i>0\},
    \end{align*}
    where ex$(n, K_k)$ is the maximum number of edges in an $n$-vertex graph that does not contain $K_k$ as a subgraph.
        \end{problem}

\medskip
%%%%%%%%%%%%%%%%%%%%%%%%%%%%%%%%
\paragraph{Funding statement}
%%%%%%%%%%%%%%%%%%%%%%%%%%%%%%%%

Yuhang Bai was supported by the National Natural Science Foundation of China (Nos.~12131013 and 12471334), China Scholarship Council (No. 202406290002), and  Shaanxi Fundamental Science Research Project for Mathematics and Physics (No. 22JSZ009). Zixuan Yang was supported by  National Natural Science Foundation of China (Nos.~12401464 and 12471334) and China Scholarship Council (No. 202406290243).

% ---------------- bib -------------------------------


\begin{thebibliography}{30}
\bibitem{Erd3} P. Erd\H{o}s, Some theorems on graphs,  \emph{Riveon Lematematika}, \textbf{9} (1955), 13-17.


\bibitem{Erd1} P. Erd\H{o}s, On a theorem of Rademacher-Tur\'{a}n,  \emph{Illinois J. Math.}, \textbf{6} (1962), 122-127.


\bibitem{Erd2} P. Erd\H{o}s,   On the number of complete subgraphs contained in certain graphs, \emph{Magy.
Tud. Acad. Mat. Kut. Int. Kozl}, \textbf{7} (1962), 459-474.


\bibitem{For}  C. M. Fortuin, P. W. Kasteleyn, and J. Ginibre, Correlation inequalities on some partially ordered sets, \emph{Comm. Math. Phys.}, \textbf{22} (1971), 89-103.




\bibitem{Har} T. Harris,  A lower bound for the critical probability in a certain percolation process, \emph{Math. Proc. Cambridge Philos. Soc. }, \textbf{1} (1960),  13-20.


\bibitem{Kat} G. O. H. Katona, G. Y. Katona, and Z. Katona, Most probably intersecting families
of subsets, \emph{Comb. Prob. Comput.}, \textbf{21} (2012), 219-227.



\bibitem{Liu} H. Liu, O. Pikhurko, and K. Staden, The exact minimum number of triangles in
graphs of given order and size,  \emph{Forum Math. Pi}, \textbf{8} (2020), e8, 144pp.


\bibitem{Lov}  L. Lov\'{a}sz and M. Simonovits, On the number of complete subgraphs of a graph ii. In Studies
in Pure Mathematics: To the Memory of Paul Tur\'{a}n, pages 459-495. Springer, 1983.



\bibitem{Man} W. Mantel,  Problem 28, soln. by H. Gouventak, W. Mantel, J. Teixeira de Mattes, F. Schuh and W.A. Wythoff, Wiskundige Opgaven,  (1907) 10:60-61.


\bibitem{Raz} A. Razborov, On the minimal density of triangles in graphs, \emph{Combin. Probab.
Comput.}, \textbf{17} (2008), 603-618.

\end{thebibliography}
\end{document}